\documentclass[a4paper,reqno,twoside]{amsart} %nosumlimits,nonamelimits,
\usepackage{amssymb}
\usepackage{latexsym}
\usepackage{amsmath}
\usepackage{mathrsfs}
\usepackage[all,2cell]{xy}
        \UseAllTwocells \SilentMatrices
\usepackage{ifthen}
\usepackage{comment}
\theoremstyle{plain}
\newtheorem{propn}{Proposition}[section]
\newtheorem{thm}[propn]{Theorem}

\newtheorem{cor}[propn]{Corollary}

\theoremstyle{definition}

\theoremstyle{remark}

\newtheorem*{rems}{Remarks}

\newcommand{\ve}{\varepsilon}

\newcommand{\varphitilde}{\wt{\varphi}}
\newcommand{\varphiol}{\ol{\varphi}}

\newcommand{\toy}{\Xi}

\newcommand{\Rbar}{\mathcal{R}}
\newcommand{\Ebar}{\mathcal{E}}
\newcommand{\Rmap}{\mathcal{R}}

\newcommand{\Al}{\mathsf{A}}

\newcommand{\Bilol}{\ol{\mathsf{B}}}
\newcommand{\Bil}{\mathsf{B}}     %use for operator spaces

\newcommand{\Hil}{\mathsf{H}}
\newcommand{\hil}{\mathsf{h}}
\newcommand{\Kil}{\mathsf{K}}
\newcommand{\kil}{\mathsf{k}}
\newcommand{\Vil}{\mathsf{V}}

\newcommand{\Dh}{\mathcal{S}_h}
\newcommand{\cDh}{\Sigma_h}

\newcommand{\Exps}{\mathcal{E}}
\newcommand{\Fock}{\mathcal{F}}
\newcommand{\FFock}{\mathcal{F}}

\newcommand{\Cstar}{C^*}

\newcommand{\coproduct}{\Delta}

\newcommand{\counit}{\epsilon}

\newcommand{\norm}[1]{\lVert #1 \rVert}

\newcommand{\kilhat}{\wh{\kil}}
\newcommand{\hilhat}{\wh{\hil}}
\newcommand{\Hilhat}{\wh{\Hil}}
\newcommand{\khat}{\wh{\kil}}
\newcommand{\hhat}{\wh{\hil}}

\newcommand{\chat}{\wh{c}}

\newcommand{\Real}{\mathbb{R}}
\newcommand{\Rplus}{\Real_+}
\newcommand{\Comp}{\mathbb{C}}
\newcommand{\Nat}{\mathbb{N}}

\newcommand{\ZZ}{\mathbb{Z}}

\newcommand{\dQS}{\Delta^{\mathrm{QS}}}

\newcommand{\cb}{{\text{\tu{cb}}}}

\newcommand{\CPstrict}{CP_{\beta}}
\newcommand{\CBuw}{CB_{\sigma}}

\newcommand{\wh}{\widehat}
\newcommand{\wt}{\widetilde}
\newcommand{\ol}{\overline}
\newcommand{\ul}{\underline}
\newcommand{\ot}{\otimes}
\newcommand{\otM}{\otimes_{\mathrm{M}}}

\newcommand{\otol}{\,\ol{\otimes}\,}
\newcommand{\otul}{\,\ul{\otimes}\,}

\newcommand{\conv}{\star}

\newcommand{\op}{\oplus}
\newcommand{\la}{\langle}
\newcommand{\ra}{\rangle}

\newcommand{\tu}{\textup}

\DeclareMathOperator{\Lin}{Lin}

\DeclareMathOperator{\id}{id}

\newenvironment{alist}
{

\begin{enumerate}}
{\end{enumerate}}

\newenvironment{rlist}
{

\begin{enumerate}}
{\end{enumerate}}

\numberwithin{equation}{section}
\begin{document}

\title
[Quantum random walk approximation] {Quantum random walk approximation \\
on locally compact quantum groups}

\author[Lindsay]{J.\ Martin Lindsay}
\address{Department of Mathematics and Statistics,
Lancaster University, Lancaster LA1 4YF, U.K.} \email{j.m.lindsay@lancaster.ac.uk}
\author[Skalski]{Adam G.\ Skalski}
\address{Institute of Mathematics of the Polish Academy of Sciences,
ul.\'Sniadeckich 8, 00-956 Warszawa, Poland} \email{a.skalski@impan.pl}


\begin{abstract}
A natural scheme is established for the approximation of quantum
L\'evy processes on locally compact quantum groups by quantum random
walks. We work in the somewhat broader context of discrete
approximations of completely positive quantum stochastic convolution
cocycles on $C^*$-bialgebras.
\end{abstract}

\maketitle

\section*{Introduction}
\label{Section: Introduction}

In~\cite{QSCC3} we developed a theory of quantum stochastic convolution cocycles on counital multiplier $C^*$-bialgebras,
extending the \emph{algebraic} theory of quantum L\'evy processes created by Sch\"urmann and coworkers
(see ~\cite{Schurmann} and references therein) %~\cite{FSc},~\cite{Uwe},)
and the \emph{topological} theory of quantum stochastic convolution cocycles on compact quantum
groups and operator space coalgebras developed by the authors
%([$\text{LS}_{0,2}$]).
(\cite{QSCC2}). % see also~\cite{QSCC0} where the algebraic and topological are compared.
Here we apply the results of \cite{QSCC3}
to introduce and analyse a straightforward scheme for the approximation of such cocycles by quantum random walks. In particular
we obtain results on Markov-regular quantum L\'evy processes on locally compact quantum semigroups, extending and strengthening
results in~\cite{FrS} for the compact case. Our analysis exploits a recent approximation theorem of Belton (\cite{Bel}), which
extends that of~\cite{Sah} (used in~\cite{FrS}). The approximation scheme closely mirrors the way in which Picard iteration
operates in the construction of solutions of quantum stochastic differential equations (see~\cite{Lgreifswald}).

The study of quantum random walks on quantum groups was intitiated by Biane in the early nineties (starting from \cite{Bi1}). Some combinatorial, probabilistic and physical interpretations can be found in Chapter 5 of \cite{Majid}.
Recent work has concentrated on discrete quantum groups and the development of (Poisson and Martin) boundary theory for quantum
random walks (see ~\cite{NeTuset} and references therein). %~\cite{Benoit},~\cite{VVa},~\cite{VVe}).
Random walks of the type
considered here and in~\cite{FrS} are discussed in~\cite{FrG} in the context of finite quantum groups. For \emph{standard}
quantum stochastic cocycles on operator algebras and operator spaces (see \cite{Lgreifswald}, and references therein), quantum
random walk approximation (\cite{LPrw}, %~\cite{Sinha},
\cite{Sah}, %~\cite{GoS}, [$\text{Be}_{1,2}$]
\cite{Bel}) has seen recent applications
in the probability theory and mathematical physics literature (e.g.~\cite{AtP},~\cite{BvH}).

\section{Preliminaries}
\label{Section: Extensions}
 In this section we briefly recall some
definitions and relevant facts about strict maps and their extensions, matrix spaces over an operator space, structure maps with
respect to a character on a $C^*$-algebra, multiplier $C^*$-bialgebras and quantum stochastic convolution cocycles; we refer to
\cite{QSCC3} for a  detailed account.

\emph{General notations.} The multiplier algebra of a $C^*$-algebra $\Al$ is denoted by $M(\Al)$ (note that in \cite{QSCC3} $\tilde{\Al}$ was used). The symbols $\otul$, $\ot$ and
$\otol$ are used respectively for linear/algebraic, spatial/minimal and ultraweak tensor products, of spaces and respecively,
linear, completely bounded and ultraweakly continuous completely bounded maps.  For a subset $S$ of a vector space $V$ we denote
its linear span by $\Lin S$.
 For a Hilbert space $\hil$, we have the ampliation
\begin{equation*}
%\label{iotahatDelta}
\iota_\hil: B(\Hil; \Kil)\to B(\Hil\ot\hil; \Kil\ot\hil), \ T\mapsto
T\ot I_\hil
\end{equation*}
where context determines the Hilbert spaces $\Hil$ and $\Kil$; we
also use the notations
\begin{equation}
\label{0.2}
 \hilhat:= \Comp\oplus\hil,\ \ \chat:= \binom{1}{c}
\text{ for } c\in\hil \text{ and } \dQS := P_{\{0\}\oplus\hil} =
\begin{bmatrix} 0 & \\ & I_{\hil}
\end{bmatrix} \in B(\hilhat)
\end{equation}
(the superscript $QS$ is there to avoid confusion with coproducts).

\subsection{Strict maps and their extensions}
If $\Al_1, \Al_2$ are $C^*$-algebras then a map $\varphi: \Al_1 \to M(\Al_2)$ is called strict if it is bounded and continuous in
the strict topology on bounded subsets. The space of all such maps is denoted $B_{\beta}(\Al_1;M(\Al_2))$.  Each  map $\phi \in
B_{\beta}(\Al_1;M(\Al_2))$ has a unique strict extension $\tilde{\varphi}:M(\Al_1) \to M(\Al_2)$. This allows the natural
composition operation: if $\psi \in B_{\beta}(\Al_2;M(\Al_3))$ for another $C^*$-algebra $\Al_3$ we define $\psi \circ \varphi:=
\tilde{\psi} \circ \varphi$. A map $\varphi \in B_{\beta}(\Al_1;M(\Al_2))$ is called \emph{preunital} if its strict extension is
unital.
  Thus,
when $\varphi$ is *-homomorphic, preunital is equivalent to nondegenerate (\cite{Lance}).

It is shown in Section 1 of \cite{QSCC3} that every completely bounded map from a $C^*$-algebra to the algebra of all bounded
operators on a Hilbert space (understood as the multiplier algebra of the algebra of compact operators) is automatically strict.

\subsection{Matrix spaces}
 For an operator space $\Vil$ in $B(\Hil; \Kil)$ and
full operator space $B = B(\hil; \kil)$, the
$(\hil,\kil)$-\emph{matrix space over} $\Vil$, denoted $\Vil\otM B$,
is
\[
\{ A \in B(\Hil\ot\hil, \Kil\ot\kil): \forall_{\omega \in B_*} \
(\id_{B(\Hil;\Kil)} \otol \omega)(A) \in\Vil \}.
\]
It is an operator space lying between $\Vil\otul B$ and $\ol{\Vil} \otol B$ which is equal to the latter when $\Vil$ is
ultraweakly closed. For any map $\varphi \in CB(\Vil_1; \Vil_2)$, between operator spaces, the map $\varphi \otul \id_B$ extends
uniquely to a completely bounded map $\varphi \otM \id_B: \Vil_1\otM B \to \Vil_2\otM B$ (\cite{LWexistence}). This construction
is compatible with strict tensor products and strict extension, in the sense described in Section 1 of \cite{QSCC3}.

%%%%%%%%%%%%%%%%%%%%%%%%%%%%%%%%%%%%%%%%%%%%%%%
%%%%%%%%%%%%%%%%%%%%%%%%%%%%%%%%%%%%%%%%%%%%%%%
\subsection{$\chi$-structure maps}
Let $(\Al,\chi)$ be a $\Cstar$-algebra with character. A
$\chi$-\emph{structure map on} $(\Al,\chi)$ is a linear map $\varphi
: \Al \to B(\hilhat)$, for some Hilbert space $\hil$, satisfying
\begin{equation*}
%\label{Equation: chi structure relation}
\varphi (a^*b) =
\varphi(a)^* \chi (b) + \chi(a)^* \varphi(b) +
\varphi (a)^* \Delta^{\tu{QS}} \varphi (b),
\end{equation*}
where $\dQS$ is given by~\eqref{0.2}.
%denotes the orthogonal projection $\left[\begin{smallmatrix} 0 & \\
%& I_{\hil}
%\end{smallmatrix}\right]
% \in B(\hilhat)$ (no relation to coproducts).
The following automatic implementability result, which is
established in~\cite{QSCC2}, is key.

\begin{thm} \label{established}
Let $(\Al,\chi)$ be a $\Cstar$-algebra with
character and let $\varphi$ be a linear map $\Al \to B(\hhat)$, for some
Hilbert space $\hil$. Then the following are equivalent.
\begin{rlist}
\item
$\varphi$ is a $\chi$-structure map.
\item
$\varphi$ is \emph{implemented} by a pair $(\pi, \xi)$ consisting of a
*-homomorphism $\pi: \Al \to B(\hil)$ and vector $\xi\in\hil$, that is
$\varphi$ has block matrix form
%\[
%\begin{bmatrix} \la \xi | \\ I_{\hil} \end{bmatrix}
%\nu( \cdot )
%\begin{bmatrix} | \xi \ra & I_{\hil} \end{bmatrix}
%\text{ where }
%\nu := \pi - \iota_\hil \circ \chi,
%\]
%in other words,
\begin{equation}
\label{Equation: block} a \mapsto
\begin{bmatrix} \gamma(a) & \la \xi | \nu(a) \\
\nu(a) | \xi \ra & \nu(a) \end{bmatrix}
%\ (a\in\Al)
\text{ where }
\gamma:= \omega_\xi \circ \nu \text{ for } \nu := \pi - \iota_\hil
\circ \chi.
\end{equation}
\end{rlist}
Moreover, if $\varphi$ is a $\chi$-structure map with such a block matrix form
then it is necessarily strict, and
$\pi$ is nondegenerate if and only if $\varphitilde (1) = 0$.
\end{thm}

%\section{Multiplier $C^*$-bialgebras}
%\label{Section: Bialgebras}
% In this section we briefly recall the
%definitions and pertinent constructions for bialgebras in the $\Cstar$- and $W^*$- categories and a universal enveloping
% operation linking the two.

%\subsection{$C^*$- and von Neumann bialgebras}

\subsection{Multiplier $C^*$-bialgebras}

A (\emph{multiplier}) \emph{$C^*$-bialgebra} is a $C^*$-algebra $\Bil$ with \emph{coproduct}, that is a nondegenerate
*-homomorphism $\coproduct: \Bil \to M(\Bil\ot\Bil)$ satisfying the coassociativity conditions
\[
(\id_\Bil \ot \coproduct)\circ \coproduct =
(\coproduct \ot \id_\Bil)\circ \coproduct.
\]
A \emph{counit} for $(\Bil,\coproduct)$ is a character $\counit$ on $\Bil$
satisfying the counital property:
\[
(\id_\Bil \ot \counit)\circ \coproduct =
(\counit \ot \id_\Bil)\circ \coproduct = \id_\Bil.
\]

%\begin{rem}
%The above definitions extend those for unital $C^*$-bialgebras, for which
%$\Biltilde = \Bil$ and $\Bil\ottilde\Bil = \Bil\ot\Bil$.

%The strict extension of a coproduct is a unital *-homomorphism and the
%strict extension of a counit is a character on $\Biltilde$.
%Note however that, in general, $(\Biltilde, \coproducttilde)$ is
%\emph{not} itself a $\Cstar$-bialgebra as the inclusion
%$\Biltilde \ot \Biltilde \subset \Bil \ottilde \Bil$
%is usually proper.
%\end{rem}

\noindent
Examples of counital $\Cstar$-bialgebras include locally
compact quantum groups in the universal setting (\cite{Kus2}), in
particular all coamenable locally compact quantum groups are
included.

Let $\Bil$ be a $\Cstar$-bialgebra.
% and let $\phi\in\Lin \CPstrict(\Bil;\Altilde)$ and
The \emph{convolute} of maps $\phi_1\in\Lin \CPstrict(\Bil;M(\Al_1))$ and $\phi_2\in\Lin \CPstrict(\Bil;M(\Al_2))$ for
$\Cstar$-algebras $\Al_1$ and $\Al_2$ is defined by composition of strict maps:
\[
\phi_1\conv\phi_2 = (\phi_1\ot\phi_2)\circ \coproduct \in \Lin \CPstrict(\Bil;M(\Al_1\ot\Al_2));
\]
the same notation is used for its strict extension. The convolution operation is easily seen to be associative.

Note that, by automatic strictness, maps $\varphi_1 \in CB(\Bil;B(\hil_1))$ and $\varphi_2 \in CB(\Bil;B(\hil_2))$, for Hilbert
spaces $\hil_1$ and $\hil_2$, may be convolved:
\begin{equation}
\label{automatic convolution}
\varphi_1\conv\varphi_2 \in
CB(\Bil;B(\hil_1\ot\hil_2)).
\end{equation}

%%%%%%%%%%%%%%%%%%%%%%%%%%%%%%%%%%%%
%%%%%%%%%%%%%%%%%%%%%%%%%%%%%%%%%%%%
%%%%%%%%%%%%%%%%%%%%%%%%%%%%%%%%%%%%

\subsection{Quantum stochastic convolution cocycles} \label{Section: QSCC}

\emph{We now fix, for the rest of the paper}, a complex Hilbert
space $\kil$ refered to as the \emph{noise dimension space} and a
counital $C^*$-bialgebra $\Bil$.

%In this section we first recall
%some standard terminology and notation for quantum stochastic
%calculus, and then summarise the results on coalgebraic QS
%differential equations and QS convolution cocycles from~\cite{QSCC3}
%which are directly relevant for us here.

%For any function $g$ with values in $\kil$ let $\ghat$ denote the
%corresponding function with values in $\kilhat$, defined by
%$\ghat(s):= \widehat{g(s)} = \binom{1}{g(s)}$. Let $\FFock$ denote
%the symmetric Fock space over $L^2(\Rplus;\kil)$ and

For $0\leq r < t \leq \infty$ we let $\FFock_{[r,t[}$ denote the symmetric Fock space over $L^2([r,t[; \kil)$ and write
$I_{[r,t[}$ for the identity operator on $\FFock_{[r,t[}$ and $\Fock$ for $\FFock_{[0,\infty[}$. Also let $\Exps$ denote the
linear span of $\{\ve(g): g\in L^2(\Rplus; \kil)\}$, where $\ve(g)$ denotes the exponential vector
$\big((n!)^{-\frac{1}{2}}g^{\ot n}\big)_{n\geq 0}$  in $\FFock$.

For $\varphi\in CB(\Bil;B(\khat))$, the coalgebraic QS differential
equation
\begin{equation*}
%\label{Equation: coalg QSDE}
dl_t = l_t \conv d\Lambda_\varphi (t),
\quad l_0 = \iota_{\FFock} \circ \counit.
\end{equation*}
has a unique \emph{form} solution, denoted $l^{\varphi}$; it is
actually a \emph{strong} solution.
\begin{comment}
It is given by
\begin{equation}
 \label{rweqn1}
l^{\varphi}_{t, \varepsilon} = \ol{l}^{\varphi}_{t,
\varepsilon}|_{\Bil} \text{ where } \ol{l}^{\varphi}_{t,
\varepsilon} = \Ebar k^{\ol{\phi}}_{t, \varepsilon} \text{ for }
\ol{\phi} = \Rbar \ol{\varphi} = \ol{\Rmap \varphi} \quad (\ve \in
\Exps),
\end{equation}
$k^{\ol{\phi}}$ being the `standard' QS cocycle generated by
$\ol{\phi}\in CB_\sigma(\ol{\Bil}; B(\khat))$,
 that is, the unique weakly regular weak
solution of the QS differential equation
\[
dk_t = k_t \circ d\Lambda_{\ol{\phi}}(t), \quad k_0 = \iota_\FFock
\]
(see~\cite{Lgreifswald}).
\end{comment}
 The process $l^{\varphi}$ is a QS
convolution cocycle on $\Bil$; moreover, conversely any
Markov-regular, completely positive, contractive, QS convolution
cocycle on $\Bil$ is of the form $l^{\varphi}$ for a unique
$\varphi\in CB(\Bil; B(\khat))$. For completely bounded processes
the cocycle relation reads as follows (after some natural
identifications are made):
\[
l_{s+t} = l_s \conv \big( \sigma_s \circ l_t \big), \quad l_0 =
\iota_\FFock \circ \counit, \quad s,t\in\Rplus,
\]
where $(\sigma_s)_{s\geq 0}$ is the semigroup of right shifts on $B(\FFock)$. \emph{Markov-regularity} means that each of the
associated convolution semigroups of the cocycle is norm-continuous. In this situation, the map $\varphi$ is referred to as the
\emph{stochastic generator} of the QS convolution cocycle.  A QS convolution cocycle $l$ is said to be completely positive,
preunital,
*-homomorphic, etc.\ if each $l_t$ has that property. Generators of such cocycles are characterised in Theorem 5.2 of
\cite{QSCC3}.

\section{Approximation by discrete evolutions}
\label{Section: Approximation}

We now show that any Markov-regular, completely positive,
contractive QS convolution cocycle on $\Bil$ may be approximated in
a strong sense by discrete completely positive evolutions, and that
the discrete evolutions may be chosen to be
*-homomorphic and/or preunital, if the cocycle is.
%This extends and strengthens
%results of~\cite{FrS} for quantum L\'evy processes on compact
%quantum semigroups.

Belton's condition (\cite{Bel}) for discrete approximation of
standard Markov-regular QS cocycles (\cite{Lgreifswald}) nicely
translates to the convolution context using the techniques developed
in~\cite{QSCC3}. We show this first. Denote by $\toy_n^{(h)}$ ($h>0,
n\in\Nat$) the injective
*-homomorphism
\[
B\big(\khat^{\ot n}\big) = B(\khat)^{\otol n} \to
B\big(\FFock_{[0,hn[}\big) \ot I_{[hn,\infty[} = \Big(\,
\ol{\bigotimes}_{j=1}^{\, n}
\,B\big(\FFock_{[(j-1)h,jh[}\big)\,\Big) \ot
%B\big(\FFock_{[0,h[}\big) \otol \cdots \otol
%B\big(\FFock_{[(n-1)h,hn[}\big) \ot
I_{[hn,\infty[}
\]
arising from the discretisation of Fock space (\cite{AtP},~\cite{Bel 1}). Thus
\[
\toy_n^{(h)}: A \mapsto D_n^{(h)} A D_n^{(h)*} \ot I_{[hn,\infty[}
\]
where
\begin{align*}
&
D_n^{(h)} :=
%\iota_{\FFock_{[hn,\infty[}} \circ
\bigotimes_{j=1}^n D_{n,j}^{(h)},
\text{ for the isometries }
\\ &
D_{n,j}^{(h)}:
\khat \mapsto \FFock_{[(j-1)h,jh[}, \
\binom{z}{c} \mapsto
\big(z, h^{-1/2}c_{[(j-1)h, jh[}, 0, 0, \cdots \big).
\end{align*}
Also write $\toy_{n,\ve}^{(h)}$ for the completely bounded map
\[
B(\kilhat^{\ot n}) \to |\FFock\ra, \ \ A \mapsto \toy_n^{(h)}(A)
|\ve\ra, \quad \text{ where } h > 0 , n\in\mathbb{N} \text{ and }
\ve\in\Exps.
\]

For a map $\Psi\in CB\big( \Vil; \Vil\otM B(\khat)\big)$, in which
$\Vil$ is a concrete operator space, its \emph{composition iterates}
$(\Psi^{\bullet n})_{n\in\ZZ_+}$ are defined recursively by
\[
\Psi^{\bullet 0} := \id_\Vil, \quad \Psi^{\bullet n}:= \big(
\Psi^{\bullet (n-1)} \otM \id_{B(\khat)} \big) \circ \Psi \in
CB\big( \Vil; \Vil\otM B(\khat^{\ot n})\big), \quad n\in\Nat.
\]
Similarly, for a map $\psi \in CB\big(\Bil; B(\khat)\big)$, its
\emph{convolution iterates} $(\psi^{\conv n})_{n\in\ZZ_+}$ are
defined by
\[
\psi^{\conv 0} := \counit, \quad \psi^{\conv n} = \psi^{\conv (n-1)}
\conv \psi \in CB\big( \Bil; B(\khat^{\ot n})\big) \quad (n\in\Nat).
\]
As usual we are viewing $B(\khat^{\ot n})$ as the multiplier algebra
of $K(\khat^{\ot n})$ here, and invoking the remark
containing~\eqref{automatic convolution}, to ensure meaning for the
convolutions.
% $\psi_{n-1} \conv \psi$.
%:= \big(\psi_{n-1} \ottilde \psi\big) \circ \coproduct$.
%In short,
%$\psi_n = \psi^{\conv n}$ for $n\in\ZZ_+$.
\begin{comment}
The two forms of iteration enjoy an easy compatibility: for $\psi
\in CB\big(\Bil; B(\khat)\big)$,
\begin{equation}
\label{Equation: ease}
\Rmap \psi^{\conv n} = \Psi^{\bullet n}
%\psi_n =
%R_{K(\khat^{\ot n})} \psi_n =
%\Psi_n
\text{ where }
\Psi :=
\Rmap  \psi.
%R_{K(\khat)} \psi.
\end{equation}
\end{comment}

We need the following block matrices, on a Hilbert space
of the form $\Hilhat$:
\begin{equation*}
%\label{Equation: Dh}
\Dh
:= \begin{bmatrix}
h^{-1/2} & \\ & I_\Hil
\end{bmatrix},
\quad h > 0,
\end{equation*}
and write $\cDh$ for the map $X \mapsto \Dh X \Dh$ on $B(\Hilhat)$.
Such a conjugation provides the correct scaling for quantum
random-walk approximation (\cite{LPrw}).

\begin{thm}
\label{Theorem: 1}
Let $\varphi\in CB(\Bil; B(\khat))$. Suppose that there is a family of maps
$(\psi^{(h)})_{0<h\leq C}$
in $CB(\Bil; B(\khat))$ for some $C>0$, satisfying
\[
\big\| \varphi - \cDh \circ \big( \psi^{(h)} - \iota_{\khat} \circ
\counit \big) \big\|_\cb \to 0
%\text{ in } CB(\Bil; B(\khat))
\text{
as } h\to 0^+.
\]
Then the convolution iterates $\{ \psi^{(h)}_n:= (\psi^{(h)})^{\conv
n} : n\in \ZZ_+\}$ \tu{(}$0 < h < C$\tu{)} satisfy
\begin{equation*}
%\label{Equation: rw approx}
\sup_{t\in[0,T]} \Big\| l^{\varphi}_{t,\ve} - \toy^{(h)}_{[t/h],\ve}
\circ \psi^{(h)}_{[t/h]}  \Big\|_\cb
%{CB(\Bil; |\FFock\ra)}
\to 0 \text{ as } h \to 0^+,
\end{equation*}
for all $T\in\Rplus$ and $\varepsilon \in \Exps$.
\end{thm}

\begin{proof}
Set $\phi := (\textup{id} \ot \varphi) \circ \coproduct$ and $\Psi^{(h)} := (\textup{id} \ot \psi^{(h)}) \circ \coproduct$.
%R^\sigma_{B(\khat)}
 Denote the enveloping von Neumann algebra of $\Bil$ by
$\Bilol$ and let $\ol{\phi}, \ol{\Psi^{(h)}} \in \CBuw\big(\Bilol; B(\khat)\big)$ denote respectively the normal extensions of
$\phi$ and $\Psi^{(h)}$ to $\Bilol$. Similarly let $\ol{\counit}:\Bilol \to \mathbb{C}$ denote the normal extension of the
counit.

As the maps transforming $\varphi$  into $\ol{\phi}$ and $\psi^{(h)}$ into $\ol{\Psi^{(h)}}$ are complete isometries (by
Proposition 2.1 and remarks after Theorem 1.2 in \cite{QSCC3}), we have
\[
\big\|
\big( \id_{\Bilol} \otol \cDh \big) \circ
%\big( 1_{\Bilol} \ot \Dh \big)
\big( \Psi^{(h)} - \iota_{\khat}  \big)
%( \cdot ) \big( 1_{\Bilol} \ot \Dh \big)
- \ol{\phi} \big\|_\cb
=
\big\|
\cDh \circ
\big( \psi^{(h)} - \iota_{\khat} \circ \counit \big)
%( \cdot )\Dh
- \varphi \big\|_\cb
\]
which tends to $0$ as $h\to 0^+$. Therefore, by Theorem 7.6 of  \cite{Bel}, it follows that
\[
\sup_{t\in[0,T]} \Big\| \big( \id_{\Bilol} \otol \toy^{(h)}_{[t/h],\ve} \big) \circ \ol{\Psi^{(h)}_{[t/h]}}
 - k^{\ol{\phi}}_{t,\varepsilon}
\Big\|_\cb
%{CB( \Bilol; \Bilol \otol | \FFock \ra )}
\to 0 \text{ as } h \to 0^+,
\]
where $\ol{\Psi_n^{(h)}}:= (\ol{\Psi^{(h)}})^{\bullet n}$ and $k^{\ol{\phi}}$ denotes the `standard' QS cocycle generated by
$\ol{\phi}$,
 that is, the unique weakly regular weak
solution of the QS differential equation
\[
dk_t = k_t \circ d\Lambda_{\ol{\phi}}(t), \quad k_0 = \iota_\FFock
\]
(see~\cite{Lgreifswald}). By the results of  Section 4 of \cite{QSCC3} $ l^{\varphi}_{t,\varepsilon} =
\ol{l}^{\varphi}_{t,\varepsilon}|_{\Bil} $ where $ \ol{l}^{\varphi}_{t,\varepsilon} =
%k^{\ol{\phi}, \counitol}_{t,\varepsilon} =
\
%E^\sigma_{|\FFock\ra}
(\ol{\counit} \ol{\ot} \id) \circ k^{\ol{\phi}}_{t,\varepsilon} $ ($t\in\Rplus, \ve\in\Exps$). The result therefore follows from
the easily checked identity $
(\ol{\counit} \ol{\ot} \id) \circ \ol{\Psi_n^{(h)}}|_\Bil = \psi^{(h)}_n $ ($n\in\ZZ_+$). % which is evident from~\eqref{Equation: ease}.
\end{proof}

\begin{comment}

\begin{rems}
(i) Since multiplicativity of the coproduct plays no role in the above proof, the proper hypothesis for the theorem is that
$\Bil$ be a (multiplier) $\Cstar$-hyperbialgebra.

(ii) The theorem of Belton used in the above proof allows the
hypothesis to be weakened to a condition on columns (like that of
the conclusion), namely $\varphi \in L \big( \khat ; CB\big(\Bil ; |
\khat \ra \big) \big)$ and, for all $c\in\kil$,
\[
\cDh \circ
\big( \psi^{(h)} - \iota_{\khat} \circ \counit \big)
( \cdot )
%\Dh
| \chat\, \ra
\to
\varphi_{\chat}
\text{ in } CB(\Bil; |\khat\ra) \text{ as } h \to 0^+.
\]
The proof of this extension requires only minor modification,
however the result we have given is adequate for our purposes here.
\end{rems}
\end{comment}

For the next two propositions coproducts play no role. Recall
Theorem~\ref{established} on the automatic implementability of
$\chi$-structure maps.
%map~\eqref{Equation: chi structure relation}.

\begin{propn}
\label{Proposition: Y}
Let $(\Al,\chi)$ be a $\Cstar$-algebra with character and let
$\varphi: \Al \to B(\hilhat)$ be a $\chi$-structure map.
Letting $(\pi,\xi)$ be an implementing pair for $\varphi$,
set
\[
U_\xi^{(h)} :=
\begin{bmatrix} c_{h,\xi} & -s_{h,\xi}^* \\
s_{h,\xi} & c_{h,\xi} Q_\xi + Q_\xi^\perp
\end{bmatrix},
\text{ for } h > 0 \text{ such that } h \|\xi\|^2\leq 1,
\]
where
\[
c_{h,\xi} := \sqrt{1 - s_{h,\xi}^* s_{h,\xi}} =
 \sqrt{1 - h \|\xi\|^2},
 s_{h,\xi} = h^{1/2} | \xi \ra,
 \text{ and } Q_\xi:= P_{\Comp\,\xi}.
% = | \xi' \ra \la \xi' |,
\]
%for $\xi' := \| \xi \|^{-1} \xi$ \tu{(}or $0$ if $\xi = 0$\tu{)}.
Then the following hold.
\begin{alist}
\item
Each $U_\xi^{(h)}$ is a unitary operator on $\hilhat$.
\item
The family of *-representations

$\rho^{(h)} = \wh{\pi}^{(h)}_{\xi}: \Bil \to B(\hilhat)$
\tu{(}$h>0,\ h\|\xi\|^2\leq 1$\tu{)}, defined by
\begin{equation*}
%\label{Equation: rho h}
%\Big(\psi^{(h)} =
\wh{\pi}^{(h)}_{\xi}(b) := U_\xi^{(h)\,*} (\chi \op
\pi)(b)U_\xi^{(h)}
%\Big)_{h>0, h\|\xi\|^2\leq 1}
\end{equation*}
satisfies
\begin{equation}
\label{Equaton: difference}
 \varphi - \cDh \circ \big( \rho^{(h)} -
\iota_{\hilhat} \circ \chi \big)
%( \cdot ) \Dh
=
\frac{h}{1 + c_{h,\xi}}\, \varphi_1 -
\frac{h^2}{( 1 + c_{h,\xi})^2}\, \varphi_2
\end{equation}
for some completely bounded maps $\varphi_1, \varphi_2: \Al \to
B(\hilhat)$ independent of $h$.
\item
Each *-representation $\rho^{(h)}$ is nondegenerate if \tu{(}and
only if\tu{)} $\pi$ is.
\end{alist}
\end{propn}
\begin{proof}
For the proof we drop the subscripts $\xi$, set
\[
Q = Q_\xi, \ c_h = c_{h,\xi}, \ s_h = s_{h,\xi} \text{ and put } d_h
:= c_h - 1.
\]

\noindent
 (a)
%Unitarity of $U_\xi^{(h)}$
This is evident from the identities
\[
c_h^* = c_h, \ c_h^2 + s_h^* s_h = 1, \ s_h^*Q^\perp = 0 \text{ and
} s_h s_h^* = \big( 1 - c_h^2 \big) Q.
\]
%\[
%h^{1/2} | \xi \ra = s(h) \big| \xi' \big\ra, \quad
%Q_\xi =
%\big| \xi' \big\ra \big\la \xi' \big|
%\text{ and }
%c(h)^2 + s(h)^2 = 1.
%\]
\noindent
 (b)
  Set $\nu = \pi - \iota_\hil \circ \chi$ and $\gamma =
\omega_\xi \circ \nu$
%\la\xi, \nu(\cdot)\xi\ra$
so that $\varphi$ has block matrix form~\eqref{Equation: block}.
%and $d_{h,\xi} := c_{h,\xi} - 1$.
Then, noting the identities
\[
d_h = - h \big( 1 + c_h \big)^{-1} \|\xi\|^2,\quad \|\xi\|^2 Q =
|\xi\ra\la\xi|, \quad c_h Q + Q^\perp = d_h Q + I_\hil,
\]
we have

\begin{align*}
&
 \rho^{(h)}(a) - \chi(a) I_{\hilhat}
\\
&=
U_\xi^{(h)\,*} \begin{bmatrix} 0 & \\ & \nu(a) \end{bmatrix} U_\xi^{(h)}
\\&{}
\\
&=
\begin{bmatrix} 0 & h^{1/2} \la\xi| \\
0 & d_h Q + I_\hil \end{bmatrix}
\begin{bmatrix} 0 & 0 \\
h^{1/2}\nu(a)|\xi\ra & d_h \nu(a)Q_\xi + \nu(a) \end{bmatrix}
\\&{}
\\
&=
\begin{bmatrix}
h\gamma(a) &
h^{1/2} \la\xi| \nu(a) \big[ d_h Q + I_\hil \big] \\
& \\
h^{1/2}\big[d_h Q_ + I_\hil\big]\nu(a) |\xi\ra & d_h^2 Q \nu(a)Q +
d_h \big( Q \nu(a) + \nu(a) Q \big) + \nu(a)
\end{bmatrix}
\\&{}
\\
&=
\begin{bmatrix}
 h^{1/2} &  \\ & I_\hil
 \end{bmatrix}
\Big(
 \varphi(a) - h\big(1 + c_h \big)^{-1}\, \varphi_1(a) +
h^2\big(1 + c_h \big)^{-2}\, \varphi_2(a)
 \Big)
\begin{bmatrix}
 h^{1/2} &  \\ & I_\hil
  \end{bmatrix}
\end{align*}
where
\[
\varphi_1 =
\begin{bmatrix} 0 & \gamma(\cdot)\la\xi| \\
\gamma(\cdot)|\xi\ra & X\nu(\cdot) + \nu(\cdot)X \end{bmatrix}
\text{ and }
\varphi_2 =  \gamma(\cdot)
\begin{bmatrix} 0 &  \\ & X \end{bmatrix}
\text{ for } X = |\xi\ra\la\xi|,
\]
from which (b) follows.

\noindent
 (c)
 This is evident from the unitality of each
$U_\xi^{(h)}$ and the fact that $\chi$ is a character.
\end{proof}

\begin{rems}
(i)
 Both $U_\xi^{(h)}$ and $\wh{\pi}^{(h)}_{\xi}$ are
norm-continuous in $h$; they converge to $I_{\hilhat}$ and
$\chi\op\pi$ respectively as $h\to 0^+$.

(ii)
 For the simplest class of $\chi$-structure map, namely
\[
\varphi =
\begin{bmatrix} 0 & \\ & \nu  \end{bmatrix},
\quad \text{ where } \nu = \pi - \iota_\hil \circ \chi \text{ for a
*-homomorphism } \pi: \Al \to B(\hil),
\]
$U_0^{(h)}= I$ and $\wh{\pi}^{(h)}_0 = ( \chi \op \pi )$ so
\[
\wh{\pi}^{(h)}_0 = \varphi + \iota_{\hilhat}, \text{ for all } h>0.
\]

 (iii)
 Unwrapping $\wh{\pi}^{(h)}_\xi(a)$:
\[
\begin{bmatrix}
\chi(a) + h\gamma(a) & &
s_{h,\xi}^*\big( \nu(a) - \frac{h \gamma(a)}{1+c_{h,\xi}} I_\hil \big) \\
\big( \nu(a) - \frac{h \gamma(a)}{1+c_{h,\xi}} I_\hil \big) s_{h,\xi} & &
\pi(a) - \frac{h}{1+c_{h,\xi}}\big(X\nu(a)+\nu(a)X\big) +
\frac{h^2\gamma(a)}{(1+c_{h,\xi})^2} X
\end{bmatrix}
\]
reveals the vector-state realisation
\[
\omega_{e_0} \circ
%\big\la
%e_0,
%\left(\begin{smallmatrix}1 \\ 0 \end{smallmatrix}\right),
%\binom{1}{0},
\wh{\pi}^{(h)}_\xi
%(\cdot) e_0
%\left(\begin{smallmatrix}1 \\ 0 \end{smallmatrix}\right)
%\binom{1}{0}
%\big\ra
=
\omega_{\Omega^{(h)}_\xi} \circ
%\big\la \Omega^{(h)}_\xi,
(\chi \op \pi),
%(\cdot)\Omega^{(h)}_\xi \big\ra,
\text{ where } \Omega^{(h)}_\xi := U_\xi^{(h)} e_0
%\left(\begin{smallmatrix}1 \\ 0 \end{smallmatrix}\right)
%\binom{1}{0},
%=
%\binom{c(h)}{h^{1/2}\xi},
\text{ and } e_0 = \binom{1}{0}\in\hilhat,
\]
for the state $\chi + h \gamma = \chi + h\|\xi\|^2\big(\omega_{\xi'}
\circ \pi - \chi\big)$
%where $\omega_{\xi,\pi} := \big\la \xi', \pi(\cdot) \xi' \big\ra$.
%for $\xi' := \|\xi\|^{-1}\xi$ (if $\xi\neq 0$).
where $\xi'$ equals $\norm{\xi}^{-1} \xi$, or $0$ if $\xi = 0$.
Indeed, finding such a representation was the strategy of proof
in~\cite{FrS}.

 (iv)
 This remark will be used in the proof of Theorem~\ref{Theorem:
discrete approximation}. If instead of being a $\chi$-structure map,
$\varphi$ is given by
\begin{equation}
\label{Equation: instead}
\begin{bmatrix} \la \xi | \\ D^* \end{bmatrix}
\nu( \cdot )
\begin{bmatrix} | \xi \ra & D \end{bmatrix}
\text{ where }
\nu := \pi - \iota_\Hil \circ \chi,
\end{equation}
for a nondegenerate representation $\pi: \Al \to B(\Hil)$, vector
$\xi\in\Hil$ and isometry $D\in B(\hil;\Hil)$ then, replacing the
unitaries $U_\xi^{(h)}$ by the isometries
$
V_{\xi,D}^{(h)}:= U_\xi^{(h)}
\left[\begin{smallmatrix} 1 & \\ & D \end{smallmatrix}\right]
\in B\big(\hilhat; \Hilhat\big)$
in the above proof yields a
family of completely positive preunital maps
 $\rho^{(h)} = \rho_{\pi,\xi,D}^{(h)}$ ($0< h,\ h\|\xi\|^2\leq 1$),
defined by
\[
\psi_{\pi,\xi,D}^{(h)} := V_{\xi,D}^{(h)\,*} (\chi \op \pi)(\cdot )
V_{\xi,D}^{(h)}
\]
satisfying~\eqref{Equaton: difference}, with completely bounded maps
\[
\varphi_1 =
\begin{bmatrix} 0 & \gamma(\cdot)\la\eta| \\
\gamma(\cdot)|\eta\ra & Y^*\nu(\cdot)D + D^*\nu(\cdot)Y \end{bmatrix}
\text{ and }
\varphi_2 =  \gamma(\cdot)
\begin{bmatrix} 0 &  \\ & X \end{bmatrix}
\]
where now $Y=|\xi\ra\la\eta|$, for $\eta = D^*\xi\in\hil$, but $X$
is still equal to $|\xi\ra \la\xi|$.
\end{rems}

\begin{propn}
\label{Proposition: X}
 Let $(\Al,\chi)$ be a $\Cstar$-algebra with
character and let $\varphi \in CB\big(\Al;B(\khat)\big)$. Suppose
that $\varphiol(1)\leq 0$ and $\varphi$ is expressible in the form
\begin{equation}
\label{Equation: decomposition}
\varphi_1-\varphi_2 \text{ where }
\varphi_1 \in CP\big(\Al;B(\khat)\big) \text{ and }
\varphi_2 =
\chi(\cdot) \big( \dQS + |\zeta\ra\la e_0| + |e_0\ra\la\zeta| \big),
\end{equation}
for a vector $\zeta\in\khat$. Then there is $C>0$ and a family of
completely positive contractions $\big(\phi^{(h)}: \Al \to
B(\kilhat)\big)_{0<h\leq C}$,
 such that
\begin{equation}
\label{Equation: dagger}
 \big\| \varphi - \cDh \circ \big(
\phi^{(h)} - \iota_{\hilhat} \circ \chi \big) \big\|_\cb
%( \cdot ) \Dh
\to
%\varphi \text{ in } CB\big(\Al; B(\khat)\big)
0 \text{ as } h\to 0^+.
\end{equation}
\end{propn}
\begin{proof}
It follows from Proposition 4.3 and Theorem 4.4 of~\cite{S}, and their
proofs, that there is a Hilbert space $\hil$ containing $\kil$ and a
$\chi$-structure map $\theta: \Al \to B(\hilhat)$ such that
$\varphi$ is the compression of $\theta$ to $B(\khat)$.
% = P \theta(\cdot) P$,
%where $P$ is the orthogonal projection onto $\khat \subset \hilhat$.
By Proposition~\ref{Proposition: Y} b, there is some $C>0$ and a family of *-homomorphisms $\big( \rho^{(h)}: \Al\to
B(\hilhat)\big)_{0<h\leq C}$ satisfying
\[
\big\| \theta - \cDh \circ \big( \rho^{(h)} - \iota_{\hilhat} \circ \chi \big) \big\|_{\tu{cb}}
%( \cdot ) \Dh
\to
%\theta \text{ in } CB\big(\Al; B(\hilhat)\big)
0 \text{ as } h\to 0^+.
\]
It follows that~\eqref{Equation: dagger} holds for the compressions
$\phi^{(h)}$
%:= P \wh{\pi}^{(h)}_{\xi}(\cdot) P: \Al \to  B(\khat)$,
of $\rho^{(h)}$ to  $B(\khat)$, which are manifestly completely
positive and contractive.
\end{proof}

%\begin{rem}
%When $\Bil$ and $\kil$ are both assumed to be separable, there is an
%alternative proof of the above result via standard QS cocycles
%(\cite{GL1}).
%\end{rem}

Combining the above results we obtain the following discrete
approximation theorem for QS convolution cocycles.

\begin{thm}
\label{Theorem: discrete approximation}
 Let $l$ be a Markov-regular,
completely positive, contractive quantum stochastic convolution cocycle on a counital $\Cstar$-bialgebra $\Bil$. Then there is
some $C>0$ and a family of completely positive contractions $\big( \psi^{(h)}:\Bil\to B(\khat)\big)_{0< h \leq C}$, such that the
convolution iterates $\big( \psi^{(h)}_n := (\psi^{(h)})^{ \conv n} \big)_{n\in\ZZ_+}$ satisfy
\[
%\begin{equation}
%\label{Equation: rw approx}
\sup_{t\in[0,T]} \Big\| l_{t,\ve} - \toy^{(h)}_{[t/h],\ve} \circ
\psi^{(h)}_{[t/h]} \Big\|_\cb
%{CB(\Bil; |\FFock\ra)}
\to 0 \text{ as
} h \to 0^+,
%\end{equation}
\]
for all $T\in\Rplus$ and $\varepsilon\in\Exps$. Moreover if $l$ is
*-homomorphic, and/or preunital, then each $\psi^{(h)}$ may be
chosen to be so too.
\end{thm}
\begin{proof}
By Theorem 5.2 a of \cite{QSCC3} we know that $l=l^\varphi$ for some $\varphi\in CB\big(\Bil; B(\khat)\big)$ which has a
decomposition of the form~\eqref{Equation: decomposition}, with $\chi = \counit$. The first part therefore follows from
Proposition~\ref{Proposition: X} and Theorem~\ref{Theorem: 1}. If $l$ is preunital then $\varphi$ may be expressed in the
form~\eqref{Equation: instead} and so, by the remark containing \eqref{Equation: instead}, it follows that the completely
positive maps $\psi^{(h)}$ may be chosen to be preunital.

Now suppose that $l$ is *-homomorphic. Then, by Theorem 5.2 c of \cite{QSCC3}, $\varphi$ is an $\counit$-structure map and so, by
Theorem ~\ref{established}, $\varphi$ has an implementing pair $(\pi,\xi)$ with $\pi$ nondegenerate if $l$ is. It therefore
follows from Proposition~\ref{Proposition: Y} that the maps $\psi^{(h)}$ may be chosen to be *-homomorphic
--- and also nondegenerate if the cocycle $l$ is nondegenerate. This completes the proof.
\end{proof}

Using the standard language of quantum L\'evy processes the last statement of the above theorem can be rephrased as follows.

\begin{cor}
Every Markov-regular quantum L\'evy process on a multiplier $C^*$-bialgebra can be approximated (in its Fock space realisation,
with respect to the pointwise-strong operator topology) by quantum random walks.
\end{cor}

%\actual deliberate mistake for compiling latex
\end{document}